\documentclass[reqno]{amsart}
\usepackage{graphicx} 
\usepackage[utf8]{inputenc}
\usepackage{amsmath}

\usepackage{tikz}

\usetikzlibrary{shapes.geometric}
\usepackage{amsthm}

\usepackage{amsfonts}
\usepackage{caption}
\usepackage{amssymb}
\usepackage{verbatim}
\usepackage{mathtools}
\usepackage{subcaption}
\usepackage{changepage}
\newtheorem{theorem}{Theorem}
\newtheorem{lemma}[theorem]{Lemma}

\newtheorem{corollary}[theorem]{Corollary}
\numberwithin{equation}{section}

\title{The Period of Ducci Cycles on $\mathbb{Z}_{2^l}$ for Tuples of Length $2^k$}

\author{Mark L. Lewis }
\address{Department of Mathematical Sciences\\
Kent State University\\
Kent, OH 44242}
\email{lewis@math.kent.edu}
\author{Shannon M. Tefft}
\address{Department of Mathematical Sciences\\
Kent State University\\
Kent, OH 44242}
\email{stefft@kent.edu}


\date{August 2024}

\begin{document}

\begin{abstract}
    Let the Ducci function $D: \mathbb{Z}_m^n \to \mathbb{Z}_m^n$ be defined as 
    \[D(x_1, x_2, ..., x_n)=(x_1+x_2 \; \text{mod} \; m, x_2+x_3 \; \text{mod} \; m, ..., x_n+x_1 \; \text{mod} \; m)\]
     and let the Ducci sequence of $\mathbf{u}$ be the sequence $\{D^{\alpha}(\mathbf{u})\}_{\alpha=0}^{\infty}$.
    In this paper, we will provide another proof that for $n=2^k$ and $m=2^l$, that all Ducci sequences will end in $(0,0,...,0)$ and additionally prove that this will happen in at most $2^{k-1}(l+1)$ iterations of $D$. 
\end{abstract}
\maketitle

\section{Introduction}\label{introduction}
\indent Consider a function $D: \mathbb{Z}_m^n \to \mathbb{Z}_m^n$, which we define as 
\[D(x_1, x_2, ..., x_n)=(x_1+x_2 \; \text{mod} \; m, x_2+x_3 \; \text{mod} \; m, ..., x_n+x_1 \; \text{mod} \; m).\]
Like \cite{Breuer,Ehrlich,Glaser}, we call $D$ the \textbf{Ducci function} and for a given tuple $\mathbf{u} \in \mathbb{Z}_m^n$, we say that the \textbf{Ducci sequence of} $\mathbf{u}$ is the sequence $\{D^{\alpha}(\mathbf{u})\}_{\alpha=0}^{\infty}$.  

\indent Looking at an example, consider $(3,1,3) \in \mathbb{Z}_4^3$ and the first few terms in its Ducci sequence: $(3,1,3), (0,0,2), (0,2,2),(2,0,2),(2,2,0),(0,2,2)$. One can see that if you continue the Ducci sequence, it cycles through the tuples $(0,2,2), (2,0,2),$ $(2,2,0)$. We say that these three tuples are the Ducci cycle of the Ducci sequence of $(3,1,3)$. More specifically, we say that the \textbf{Ducci cycle} of a tuple $\mathbf{u}$ is the set $\{\mathbf{v} \mid \exists \alpha \in \mathbb{Z}^+ \cup \{0\}, \beta \in \mathbb{Z}^+  \ni \mathbf{v}=D^{\alpha+\beta}(\mathbf{u})=D^{\alpha}(\mathbf{u})\}$, a term \cite{Breuer,Chamberland, Furno} also uses. We define $\mathbf{Len(u)}$ to be the smallest $\alpha$ that can be found to satisfy the equation $\mathbf{v}=D^{\alpha+\beta}(\mathbf{u})=D^{\alpha}(\mathbf{u})$ from our definition for some $\beta \in \mathbb{Z}^+, \mathbf{v} \in \mathbb{Z}_m^n$. We also define $\mathbf{Per(u)}$ to be the smallest $\beta$ that satisfies the same equation. We obtain the notations of $Len(\mathbf{u})$ and $Per(\mathbf{u})$ from Definition 1 of \cite{Breuer}, with \cite{Breuer2, Burmester, Ehrlich} being some other sources that call this the period. Notice that $Per(\mathbf{u})$ is also the number of distinct tuples in the Ducci cycle of $\mathbf{u}$. If a tuple $\mathbf{v} \in \mathbb{Z}_m^n$ is in the Ducci cycle of the Ducci sequence for some tuple $\mathbf{u}$, we may simply say that $\mathbf{v}$ is in a Ducci cycle. Because $\mathbb{Z}_m^n$ is finite, every Ducci sequence enters a cycle.

\indent Define $K(\mathbb{Z}_m^n)$ to be the set of all tuples in $\mathbb{Z}_m^n$ that are in the Ducci cycle for the Ducci sequence of some $\mathbf{u} \in \mathbb{Z}_m^n$,which is given in Definition 4 of \cite{Breuer}. The following theorem about $K(\mathbb{Z}_m^n)$ is stated on page 6001 of \cite{Breuer} but we will provide a proof of it in Section \ref{Findings_General}:
\begin{theorem}\label{cycle_subgroup_thm}
  $K(\mathbb{Z}_m^n)$ is a subgroup of $\mathbb{Z}_m^n$. 
\end{theorem}
\indent If for a tuple $\mathbf{u}$ there exists a tuple $\mathbf{v}$ such that $D(\mathbf{v})=\mathbf{u}$, we say that $\mathbf{v}$ is a \textbf{predecessor} of $\mathbf{u}$. We believe page 259 of \cite{Furno} is the first time this term is used, then later by \cite{Breuer, Glaser}.

\indent It is not always the case that the Ducci sequence ends in a cycle containing more than one tuple. Notice that the Ducci cycle of $(0,0,...,0)$ consists only of itself and that it satisfies $Per(\mathbf{u})=1$. In addition to this, $(0,0,...,0)$ is the only tuple that satisfies $D(\mathbf{u})=\mathbf{u}$, which we believe is first noted in Remark 5 of \cite{Burmester} and then again by \cite{Breuer}. This means that only $(0,0,...,0)$ and tuples whose Ducci cycle are $\{(0,0,..,0)\}$ have a period of $1$. We say that a Ducci sequence \textbf{vanishes} when its cycle is $\{(0,0,...,0)\}$, a term first used by \cite{Burmester} on page 117 and also by \cite{Breuer, Glaser}. 

\indent An important Ducci sequence is the one corresponding to $(0,0,...,0,1) \in \mathbb{Z}_m^n$. We call the Ducci sequence of this tuple the \textbf{basic Ducci sequence} of $\mathbb{Z}_m^n$, like \cite{Ehrlich} first uses on page 302, as well as \cite{Breuer,Dular, Glaser}. In a similar way to particularly \cite{Breuer} in Definition 5 and to \cite{Ehrlich, Glaser}, we denote $L_m(n)=Len(0,0,...,0,1)$ and $P_m(n)=Per(0,0,...,0,1)$. These are significant because Lemma 1 in \cite{Breuer} tells us that for every $\mathbf{u} \in \mathbb{Z}_m^n$, $L_m(n) \geq Len(\mathbf{u})$ and $Per(\mathbf{u})|P_m(n)$. Page 858 of \cite{Breuer2} also uses the notation of $P_m(n)$ to mean the maximal period for a Ducci sequence on $\mathbb{Z}_m^n$.


\indent In this paper, we focus on the case where $m=2^l$ and $n=2^k$ for integers $k,l \geq 1$. It is first proven in (I) on page 103 of \cite{Wong} that all Ducci sequences in $\mathbb{Z}_{2^l}^{2^k}$ vanish. Theorem 2.3 of \cite{Dular} also provides a proof of this. However, we would like to establish a value for $L_m(n)$ in this case. Our main goal is to prove the following theorem:
\begin{theorem}\label{MainThm}
    Let $n=2^k$, $m=2^l$ for integers $l,k \geq 1$. Then $L_m(n)=(l+1)2^{k-1}$ and $P_m(n)=1$.
\end{theorem}
As part of his proof, \cite{Wong} finds that $D^{l*2^k}(\mathbf{u})=(0,0,...,0)$ for every tuple $\mathbf{u} \in \mathbb{Z}_{2^l}^{2^k}$, however, neither \cite{Wong} nor \cite{Dular} find a value for $L_{2^l}(2^k)$. 

\indent It is worth noting that as a result of  all Ducci sequences vanishing, we gain the following corollary:
\begin{corollary}
$K(\mathbb{Z}_{2^l}^{2^k})$ is the trivial subgroup.
\end{corollary}

\indent The work in this paper was done while the second author was a Ph.D. student at Kent State University under the advisement of the first author and will appear as part of the second author's dissertation.
\section{Background}\label{Background}
\indent There are many different versions of the Ducci function. The first and most common Ducci function is $\bar{D}: (\mathbb{Z}^+ \cup \{0\})^n \to (\mathbb{Z}^+ \cup \{0\})^{n}$ being defined as $\bar{D}(x_1, x_2, ..., x_n)=(|x_1-x_2|, |x_2-x_3|, ..., |x_n-x_1|)$, with \cite{Ehrlich, Freedman,Glaser, Furno} being a few examples, or similarly defined on $\mathbb{Z}^n$ like in \cite{Breuer}. Other sources, such as \cite{Brown,Chamberland, Schinzel}, define it on $\mathbb{R}^n$. Note that if $\mathbf{u} \in \mathbb{Z}^n$, then $\bar{D}(\mathbf{u}) \in (\mathbb{Z}^+ \cup \{0\})^n$. Therefore, for simplicity, we will refer to the Ducci case on $(\mathbb{Z}^+ \cup \{0\})^n$ or $\mathbb{Z}^n$ as the Ducci case on $\mathbb{Z}^n$.


\indent For both Ducci on $\mathbb{Z}^n$ and Ducci on $\mathbb{R}^n$, it has been proven that all Ducci sequences enter a cycle. Some examples of sources that explain why this happens on $\mathbb{Z}^n$ are \cite{Burmester, Ehrlich, Glaser, Furno} and \cite{Schinzel} proves it for $\mathbb{R}^n$ in Theorem 2. For Ducci on $\mathbb{Z}^n$, when looking at the tuples in a Ducci cycle, all of the entries belong to $\{0,c\}$ for some $c \in \mathbb{Z}^+$, which is proved in Lemma 3 of \cite{Furno} and then again later by \cite{Pompili}. Because $\bar{D}(\lambda \mathbf{u})=\lambda \bar{D}(\mathbf{u})$ when $\lambda \in \mathbb{Z}$ and $\mathbf{u} \in \mathbb{Z}^n$, this means that for this case, one only needs to worry about when Ducci is defined on $\mathbb{Z}_2$ when examining Ducci cycles and periods, something that \cite{Breuer,Ehrlich,Glaser,Furno} all talk about. For the $\mathbb{R}^n$ case, \cite{Schinzel} proves a similar result in Theorem 1: if the Ducci sequence reaches a limit point, then all of the entries in the tuple belong to $\{0,c\}$ for some $c \in \mathbb{R}$.

\indent  It has been proven that the Ducci sequence will eventually vanish for all tuples in $\mathbb{Z}^n$ if and only if if $n$ is a power of $2$. According to \cite{Brown1, Chamberland}, \cite{Ciamberlini} is the first to provide a proof of this with other papers, like \cite{Glaser, Furno}, also citing \cite{Ciamberlini} for a proof of this fact. \cite{Ciamberlini} is also the first paper published that talks about Ducci sequences. Unfortunately, we are unable to find a copy of \cite{Ciamberlini}; a review of the paper can be found at \cite{CiamberliniReview}. Many different proofs for the fact that all Ducci sequences vanish have been made since \cite{Ciamberlini}, with \cite{Ehrlich, Freedman, Miller, Pompili} being a few.
As discussed in the introduction, this conclusion can be extended to $m=2^l$ where $l \geq 1$. Our ultimate goal is to establish how long we can expect it to take for a sequence to reach $(0,0,...,0)$.

\indent For the Ducci case on $\mathbb{Z}_2^n$, there have been a few sources that look at the maximum value of the length of a Ducci sequence, $L_2(n)$, for a variety of values for $n$. The first source is \cite{Ehrlich} on page 303 who found that for odd $n$, $L_2(n)=1$. The next example is the specific case where $n=2^{k_1}+2^{k_2}$ for $k_1>k_2 \geq 0$ by \cite{Glaser} in Theorem 6, who finds $L_2(n)=2^{k_2}$. 
Theorem 4 of \cite{Breuernote} then extends this to all even $n$, finding
that if $n=2^kn_1$ where $n_1$ is odd, then $L_2(n)=2^k$. This supports the case where $l=1$ for Theorem \ref{MainThm}. 

\indent Returning to our Ducci case, \cite{Wong} is the first paper to look at Ducci sequences on $\mathbb{Z}_m^n$, followed by \cite{Breuer, Breuer2, Dular}. 

\indent We start by discussing more about our example of the Ducci cycle of $(3,1,3)$ to provide a better visualization of what Ducci sequences in $\mathbb{Z}_m^n$ look like. We can do this by creating a transition graph that maps all of the Ducci sequences of $\mathbb{Z}_4^3$ and direct our focus to the connected component containing $(3,1,3)$, which can be found in Figure \ref{n_odd_figure}.

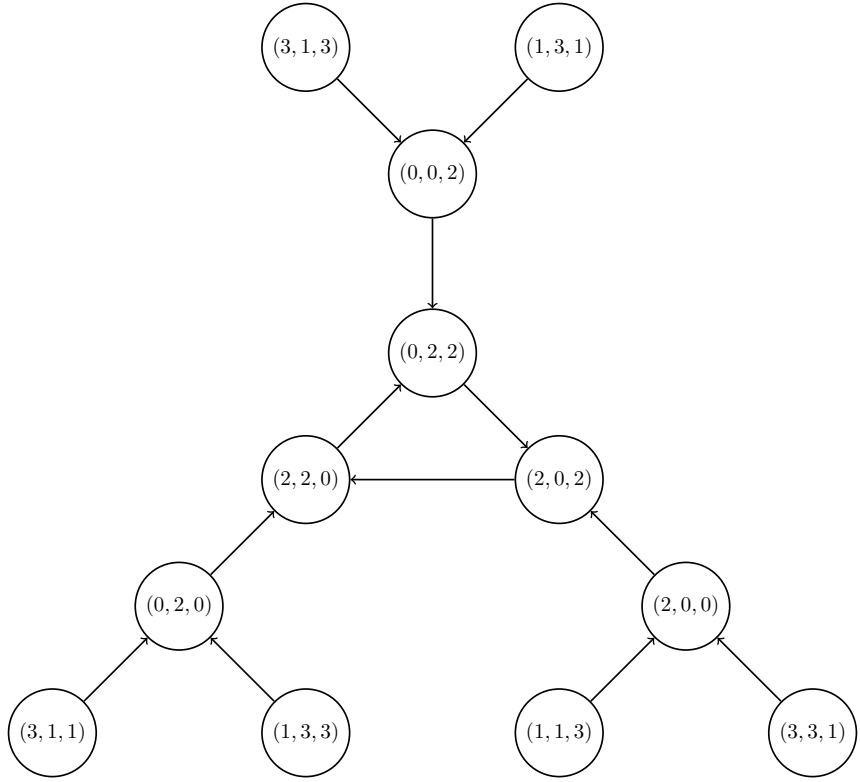
\begin{figure}
\centering
\resizebox{.9\textwidth}{!}{
\begin{tikzpicture}[node distance={30mm}, thick, main/.style = {draw, circle}]
    \node[main](1){$(0,2,2)$};
    \node[main](2)[above of=1]{$(0,0,2)$};
    \node[main](3)[above left of=2]{$(3,1,3)$};
    \node[main](4)[above right of=2]{$(1,3,1)$};
    \node[main](5)[below left of= 1]{$(2,2,0)$};
     \node[main](9)[below left of=5]{$(0,2,0)$};
    \node[main](6)[below left of=9]{$(3,1,1)$};
    \node[main](7)[below right of=9]{$(1,3,3)$};
    \node[main](8)[below right of=1]{$(2,0,2)$};
    \node[main](10)[below right of=8]{$(2,0,0)$};
    \node[main](11)[below left of=10]{$(1,1,3)$};
    \node[main](12)[below right of=10]{$(3,3,1)$};

    \draw[->](2)--(1);
    \draw[->](3)--(2);
    \draw[->](4)--(2);
    \draw[->](5)--(1);
    \draw[->](9)--(5);
    \draw[->](6)--(9);
    \draw[->](7)--(9);
    \draw[->](8)--(5);
    \draw[->](1)--(8);
    \draw[->](10)--(8);
    \draw[->](11)--(10);
    \draw[->](12)--(10);
\end{tikzpicture}
}
\caption{Transition Graph for $\mathbb{Z}_4^3$}\label{n_odd_figure}
\end{figure}

Using our definitions from before, one can see  $Len(3,1,3)=2$, $Len(0,0,2)=1$, and $Len(0,2,2)=0$. Although $(1,3,1)$ is not in the Ducci sequence of $(3,1,3)$, we can see that the Ducci sequences of $(3,1,3)$ and $(1,3,1)$ have the same Ducci cycle, as well as every other tuple, in this connected component. In addition to this, all of these tuples have period $3$.

\indent Notice that $(0,0,2)$ has two predecessors, $(3,1,3)$ and $(1,3,1)$. But not all tuples have predecessors; both $(3,1,3)$ and $(1,3,1)$ are examples of this. 

\indent We also present the transition graph for $\mathbb{Z}_4^2$ to give an example where all Ducci sequences vanish in Figure \ref{n_even_figure}.
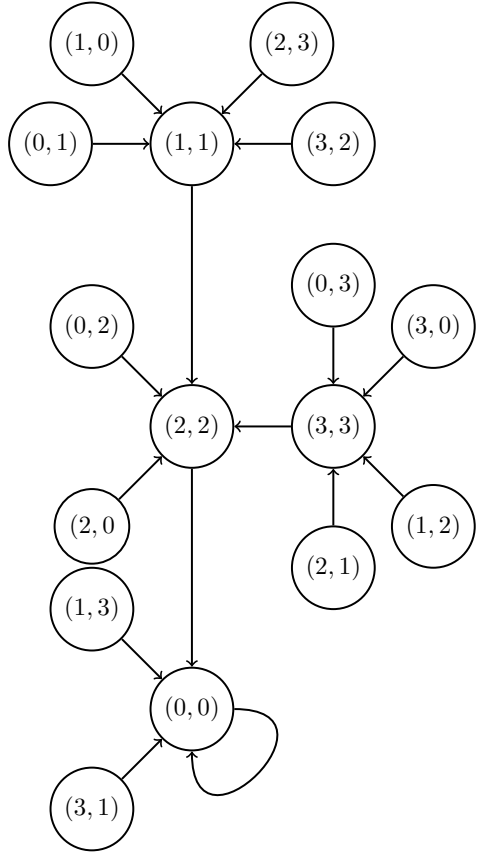
\begin{figure}
\centering
    \resizebox{.5\textwidth}{!}{
\begin{tikzpicture}[node distance={20mm}, thick, main/.style = {draw, circle}]
\node[main](1){$(0,0)$};
\node[main](2)[above left of =1]{$(1,3)$};
\node[main](3)[below left of=1]{$(3,1)$};
\node(17)[above of=1]{};
\node[main](4)[above of=17]{$(2,2)$};
\node[main](5)[above left of=4]{$(0,2)$};
\node[main](6)[below left of =4]{$(2,0$};
\node[main](7)[right of=4]{$(3,3)$};
\node[main](8)[above of=7]{$(0,3)$};
\node[main](9)[above right of=7]{$(3,0)$};
\node[main](10)[below right of=7]{$(1,2)$};
\node[main](11)[below of=7]{$(2,1)$};
\node(18)[above of=4]{};
\node[main](12)[above of=18]{$(1,1)$};
\node[main](13)[left of=12]{$(0,1)$};
\node[main](14)[above left of=12]{$(1,0)$};
\node[main](15)[above right of=12]{$(2,3)$};
\node[main](16)[right of=12 ]{$(3,2)$};

\draw[->] (1) to [out=0,in=270,looseness=5] (1);
\draw[->](2)--(1);
\draw[->](3)--(1);
\draw[->](4)--(1);
\draw[->](5)--(4);
\draw[->](6)--(4);
\draw[->](7)--(4);
\draw[->](8)--(7);
\draw[->](9)--(7);
\draw[->](10)--(7);
\draw[->](11)--(7);
\draw[->](12)--(4);
\draw[->](13)--(12);
\draw[->](14)--(12);
\draw[->](15)--(12);
\draw[->](16)--(12);
    
\end{tikzpicture}
}
\caption{Transition graph for $\mathbb{Z}_4^2$}\label{n_even_figure}
\end{figure}
For this graph, all tuples in $\mathbb{Z}_4^2$ are part of the same component. Here, we can see that $L_2(4)=3$, but not all tuples without predecessors have this length. For example, $(2,3)$ and $(1,3)$ both do not have predecessors, but $Len(2,3)=3$ and $Len(1,3)=1$.\\
\indent Notice that all tuples in $\mathbb{Z}_4^2$ that have a predecessor have exactly four predecessors. This is not a coincidence, in fact,
\begin{theorem}\label{n_odd_m_even_typesofpreds}
Let $n$ be even. If a tuple $\mathbf{u}=(x_1, x_2, ..., x_n) \in \mathbb{Z}_m^n$ has a predecessor $(y_1, y_2, ..., y_n)$, then $(y_1+z,y_2-z,...,y_{n-1}+z,y_n-z)$ is also a predecessor of $\mathbf{u}$ where $z \in \mathbb{Z}_m$. Moreover, if a tuple $\mathbf{u}$ has a predecessor, it has exactly $m$ predecessors.
\end{theorem}
\begin{proof}
Let $x_i,y_i \in \mathbb{Z}_m$. Clearly, if $(y_1, y_2, ..., y_n)$ is a predecessor to $(x_1,x_2,...,x_n)$, then so is $(y_1+z, y_2-z, ..., y_n-z)$ if $z \in \mathbb{Z}_m$. This gives us at least m predecessors to $(x_1,x_2,...,x_n)$, each beginning with a different element in $\mathbb{Z}_m$.

\indent We now confirm that these are all the predecessors of $(x_1,x_2,...,x_n)$. If we consider two tuples that are predecessors to $(x_1,x_2,...,x_n)$ that both begin with the same element, say $(y_1,y_2,...,y_n)$ and $(y_1,y_2',...,y_n')$, then we have the relations 
\[y_1+y_2 \equiv x_1 \; \text{mod} \; m\]
and
\[y_1+y_2' \equiv x_1 \; \text{mod} \; m,\]
which gives us $y_2=y_2'$. We can repeat this to see $y_i=y_i'$ for every $i, 1 \leq i \leq n$ and to see that $(x_1,x_2,...,x_n)$ only has one unique predecessor beginning with $y_1$, so it has exactly $m$ predecessors.
 \end{proof}
 \indent Let's now draw our attention back to $D$ specifically. Because 
\[D((x_1,x_2, ..., x_n)+(x_1', x_2', ..., x_n'))=D(x_1+x_1', x_2+x_2', ..., x_n+x_n')\]
\[=(x_1+x_2+x_1'+x_2', x_2+x_3+x_2'+x_3',..., x_n+x_1+x_n'+x_1')\]
\[=(x_1+x_2, x_2+x_3, ..., x_n+x_1)+(x_1'+x_2', x_2'+x_3', ..., x_n'+x_1')\]
\[=D(x_1, x_2, ..., x_n)+D(x_1', x_2', ..., x_n'),\]
it is true that $D \in End(\mathbb{Z}_m^n)$. Notice also that it remains true that $D(\lambda \mathbf{u})=\lambda D(\mathbf{u})$ when $\lambda \in \mathbb{Z}_m$ and $\mathbf{u} \in \mathbb{Z}_m^n$. Since this means $D^{\alpha}(\lambda \mathbf{u})=\lambda D^{\alpha}(\mathbf{u})$ for $\alpha \in \mathbb{Z}^+$, the Ducci sequence for $\lambda\mathbf{u}$ is $\{\lambda D^{\alpha}(\mathbf{u})\}_{\alpha=0}^{\infty}$ and if $\mathbf{u} \in K(\mathbb{Z}_m^n)$, then so is $\lambda \mathbf{u}$.\\
\indent Define $H: \mathbb{Z}_m^n \to \mathbb{Z}_m^n$ as
\[H(x_1, x_2, ..., x_n)=(x_2, x_3, ..., x_n, x_1).\]
We also have that $H \in End(\mathbb{Z}_m^n)$  and $D=I+H$ where $I$ is the identity endomorphism. Note that
\[H(D(x_1, x_2, .., x_n))=H(x_1+x_2, x_2+x_3, ..., x_n+x_1)\]
\[=(x_2+x_3, x_3+x_4, ..., x_n+x_1, x_2+x_1)\]
\[=D(x_2, x_3, .., x_n,x_1)\]
\[=D(H(x_1, x_2, ..., x_n)),\]
and therefore $H,D$ commute. This means that if 
$\{D^{\alpha}(\mathbf{u})\}_{\alpha=0}^{\infty}$ is the Ducci sequence for $\mathbf{u} \in \mathbb{Z}_m^n$, then $\{H^{\beta}(D^{\alpha}(\mathbf{u}))\}_{\alpha=0}^{\infty}$ is the Ducci sequence for $H^{\beta}(\mathbf{u})$ when $0 \leq \beta \leq n-1$. Consequently, if $\mathbf{u} \in K(\mathbb{Z}_m^n)$, then so is $H^{\beta}(\mathbf{u})$.\\
\section{Findings on Ducci for $n,m$ Arbitrary}\label{Findings_General}
\indent We can now prove Theorem \ref{cycle_subgroup_thm} from Section \ref{introduction}. Note that this theorem is stated as a fact in \cite{Breuer}, but a proof is provided here for completeness.

\begin{proof}[Proof of Theorem \ref{cycle_subgroup_thm}]
 \indent We start by showing that 
 \[D^{\alpha}((x_1, x_2, ..., x_n)+(x_1', x_2', ..., x_n'))=D^{\alpha}(x_1, x_2, ..., x_n)+D^{\alpha}(x_1', x_2', ..., x_n')\]
  by induction for ${\alpha} \in \mathbb{Z}^+$. This is true for ${\alpha}=1$ because $D$ is an endomorphism.

 Assume that it is true for $\alpha-1$, then 
    \[D^{\alpha}((x_1, x_2, ..., x_n)+(x_1', x_2', ..., x_n'))
    =D(D^{\alpha-1}((x_1, x_2, ..., x_n)+(x_1', x_2', ..., x_n')).\]
   By induction, we have 
    \[D(D^{\alpha-1}(x_1, x_2, ..., x_n)+D^{\alpha-1}(x_1', x_2', ..., x_n'))\]
    and
    \[D(D^{\alpha-1}(x_1, x_2, ..., x_n))+D(D^{\alpha-1}(x_1', x_2', ..., x_n')),\]
    which equals $D^{\alpha}(x_1, x_2, ..., x_n)+D^{\alpha}(x_1', x_2', ..., x_n')$.

    Now suppose that $(x_1, x_2, ..., x_n),(x_1', x_2', ..., x_n') \in K(\mathbb{Z}_m^n)$. If $d=P_m(n)$, then $D^d(x_1, x_2, ..., x_n)=(x_1, x_2, ..., x_n)$ and $D^d(x_1', x_2', ..., x_n')=(x_1', x_2', ..., x_n')$. Therefore, 
    \[D^d((x_1, x_2, ..., x_n)+(x_1', x_2', ..., x_n'))=D^d(x_1, x_2, ..., x_n)+D^d(x_1', x_2', ..., x_n'),\]
    which equals
    $(x_1, x_2, ..., x_n)+(x_1', x_2', ..., x_n')$
    and thus 
    \[((x_1, x_2, ..., x_n)+(x_1', x_2', ..., x_n'))\in K(\mathbb{Z}_m^n).\] 
    It follows then that $K(\mathbb{Z}_m^n) \leq \mathbb{Z}_m^n$.
    \end{proof}
    \indent When examining a Ducci sequence, it is useful to be able to know what $D^r(\mathbf{u})$ is given $\mathbf{u} \in \mathbb{Z}_m^n$. To do this,  define $a_{r,s},r,s \in \mathbb{Z}$, $r \geq 0, 1 \leq s \leq n$ so that $D^r(0,0,...,0,1)=(a_{r,n}, a_{r, n-1}, ..., a_{r,1})$. Now, any tuple $(x_1, x_2, ..., x_n) \in \mathbb{Z}_m^n$ can be written as 
    \[(x_1, x_2, ..., x_n)=\sum_{s=1}^n x_sH^{-s}(0,0,..,0,1),\]
    which gives
    \[D^r(x_1, x_2, ..., x_n)=\sum_{s=1}^n x_sH^{-s}(a_{r,n}, a_{r,n-1}, ..., a_{r,1})\]
    or
    \[\sum_{s=1}^n x_s(a_{r,s}, a_{r,s-1}, ..., a_{r,s+1}).\]
    Using this, we can now define $a_{r,s}$ as the coefficient on $x_{s-i+1}$ in the $i$th coordinate of $D^r(x_1, x_2, ..., x_n)$, with 
    \[a_{0,s}=
    \begin{cases}
        0 & s \neq 1\\
        1 & s=1
    \end{cases}.\]
    We now aim to prove a few more observations about $a_{r,s}$.
    \begin{theorem}\label{BigCoeffThm}
   Let $r \geq 1$, $1 \leq s \leq n$
        \begin{enumerate}
         \item $a_{r,s}=a_{r-1,s}+a_{r-1,s-1}$.\\
        \item For $r<n$, $a_{r,s}=\displaystyle{\binom{r}{s-1}}$ .\\
        \item $a_{r+t,s}=\displaystyle{\sum_{i=1}^n a_{t,i}a_{r,s-i+1}}$ where $t \geq 1$.\\  
        \end{enumerate}
    \end{theorem}
    \begin{proof}         
          \textbf{(1)}: 
 Letting $i=1$, we only need to show $a_{r,s}$ as the coefficient on $x_s$ in the first entry of $D^r(x_1, x_2, ..., x_n)$. Note that by how $a_{r,s}$ is defined, $D^r(x_1, x_2, ..., x_n)$ is
   \[D((\sum_{s=1}^n x_s H^{-s}(a_{r-1,n}, a_{r-1,n-1}, ..., a_{r-1,1})).\]
    The first entry of this tuple is
    \[(a_{r-1,1}+a_{r-1,n})x_1+(a_{r-1,2}+a_{r-1,1})x_2+ \cdots +(a_{r,n}+a_{r,n-1})x_n.\]
    By definition,$a_{r,s}$ is the coefficient on $x_s$ in the first entry, which the above tells us is also $a_{r-1,s}+a_{r-1,s-1}$ and (1) follows. 
    
    \indent We prove (2) and (3) by induction.
    
     \textbf{(2)}: \textbf{Basis Step} $\mathbf{r=0}$: This follows from (1) and 
      \[a_{0,s}=
    \begin{cases}
        0 & s \neq 1\\
        1 & s=1
    \end{cases}.\]
    
    \textbf{Inductive Step:} Note that if $s-1>r$, we use the convention that 
    \[a_{r,s}=\binom{r}{s-1}=0.\] Assume $a_{r-1,s}=\displaystyle{\binom{r-1}{s-1}}$ when $r<n$. By (1), we have that for $1<s \leq n$,  
    \[a_{r,s}=a_{r-1,s}+a_{r-1,s-1},\]
 which by induction is 
    \[\binom{r-1}{s-1}+\binom{r-1}{s-2}\]
     or
    \[\binom{r}{s-1}.\]
    For $s=1$, we have 
    \[a_{r,1}=a_{r-1,1}+a_{r-1,n}.\]
   By induction, this is
    \[\binom{r-1}{0}+\binom{r-1}{n-1}\]
     which is $1$ or $\displaystyle{\binom{r}{0}}$ as long as $r<n$. 
    
    \textbf{(3)}: We prove this via induction on $t$. 
    
\textbf{Basis Steps} $\mathbf{t=1,2}$: Calculating $a_{r,s}$ in terms of $a_{r-1,i}$ terms, 
\[a_{r,s}=a_{r-1,s}+a_{r-1,s-1},\]
or breaking it down further in terms of $a_{r-2,i}$,
\[a_{r,s}=a_{r-2,s}+2a_{r-2,s-1}+a_{r-2,s-2}.\]
This is
$a_{2,1}a_{r-2,s}+a_{2,2}a_{r-2,s-1}+a_{2,3}a_{r-2,s-2}$ and the basis case follows.

\textbf{Inductive Step:}
Assume that for $t'<t$, $a_{r+t',s}=\displaystyle{\sum_{i=1}^na_{t',i}a_{r,s-i+1}}$. Then calculating $a_{r+t,s}$, we have
\[a_{r+t,s}=\sum_{i=1}^n a_{t-1,i}a_{r+1,s-i+1} \, .\]
Breaking down $a_{r,s-i+1}$, this is
\[\sum_{i=1}^na_{t-1,i}(a_{r,s-i+1}+a_{r,s-i}) \,.\]
Distributing $a_{t-1,i}$ and breaking the sum up, this is
\begin{equation}\label{Equation_1st_ars_lemma}
\sum_{i=1}^na_{t-1,i}a_{r,s-i+1}+\sum_{i=1}^na_{t-1,i}a_{r,s-i} \, .
\end{equation}
Note that for the sum $\displaystyle{\sum_{i=1}^na_{t-1,i}a_{r,s-i}}$ from Expression (\ref{Equation_1st_ars_lemma}), we have
\[\sum_{i=1}^na_{t-1,i}a_{r,s-i}=\sum_{i=2}^{n+1}a_{t-1,i-1}a_{r,s-i+1} \, .\]
Because the $s$ coordinate of $a_{r,s}$ is reduced modulo $n$, this is
\[\sum_{i=1}^na_{t-1,i-1}a_{r,s-i+1}\, .\]
 Therefore, Expression (\ref{Equation_1st_ars_lemma}) is equal to
\[\sum_{i=1}^n(a_{t-1,i}+a_{t-1,i-1})a_{r,s-i+1}\]
or
\[\sum_{i=1}^na_{t,i}a_{r,s-i+1}\, .\]
(3) follows from here.
    \end{proof}
    We now prove some other useful facts about the $a_{r,s}$ coefficients that we will need later.
     \begin{corollary}
        \[a_{n,1}=
        \begin{cases}
             \displaystyle{\binom{n}{s-1}} & s \neq 1\\
            2 & s=1
        \end{cases}.\]
    \end{corollary}
    \begin{proof}
       Let $s>1$. Then 
      $a_{n,s}=a_{n-1,s}+a_{n-1,s-1}$ gives us
      \[a_{n,s}=\binom{n-1}{s-1}+\binom{n-1}{s-2}\]
      or $\displaystyle{\binom{n}{s-1}}$ by Theorem \ref{BigCoeffThm}
      
    For $s=1$, 
    $a_{n,1}=a_{n-1,1}+a_{n-1,n}$ gives us
    \[\binom{n-1}{0}+\binom{n-1}{n-1}\]
    or $2$ by Theorem \ref{BigCoeffThm}.
    \end{proof}
    We have one last lemma about the $a_{r,s}$ that is true for all $n,m$:
    \begin{lemma}\label{a_rs_when_=}
    For $r \geq 0$,
        \[a_{r,s}=a_{r,r-s+2} \, .\]
    \end{lemma}
    \begin{proof}
       We work by induction on $r$.
       
\textbf{Basis Step:} We take advantage of Theorem \ref{BigCoeffThm} to see that when $r<n$, 
\[a_{r,s}=\binom{r}{s-1}=\binom{r}{r-s+1}=a_{r,r-s+2}\, .\]

\textbf{Inductive Step:} Suppose that the theorem is true for $r-1$. Then we have 
\[a_{r-1,s}=a_{r-1,r-s+1}\]
 and
  \[a_{r-1,s-1}=a_{r-1, r-s+2}\, .\]
   Then because $a_{r,s}=a_{r-1,s}+a_{r-1,s-1}$, 

\[a_{r,s}=a_{r-1,r-s+1}+a_{r-1,r-s+2},\]

which equals $a_{r,r-s+2}$.
    \end{proof}

\section{All Tuples Vanish in $\mathbb{Z}_{2^l}^{2^k}$}\label{Main_Result}
Recall that it is our main goal to prove that for $m=2^l, n=2^k$, then $L_m(n)=(l+1)2^{k-1}$. In order to do this, we first aim to prove a number of lemmas, starting with a few lemmas that explore the value of certain binomial coefficients. We believe that Lemmas \ref{middlebinomcoeff}-\ref{middlebinomcoeffpt2}, and \ref{binomcoeff_3mod4} are known but we are including the following proofs for the sake of completeness. Throughout these proofs, we rely on the well known fact that $\displaystyle{\binom{2^j}{t}}$ is even when $t \neq 1,2^j$, a proof of which can be found in Theorem 3 of \cite{Fine}.
\begin{lemma}\label{middlebinomcoeff}
When $j \geq 2$
    \[\binom{2^j}{2^{j-1}}\equiv 2 \; \text{mod} \; 4.\]
\end{lemma}
\begin{proof}
By the Chu-Vandermonde identity [9], $\displaystyle{\binom{2^j}{2^{j-1}}}$ is
    \[\sum_{i=0}^{2^{j-1}} \binom{2^{j-1}}{i}^2=\binom{2^{j-1}}{0}^2+\binom{2^{j-1}}{2^{j-1}}^2+\sum_{i=1}^{2^{j-1}-1}\binom{2^{j-1}}{i}^2,\]
    which is congruent to $ 2 \; \text{mod} \; 4.$
\end{proof}
For the other binomial coefficients $\displaystyle{\binom{2^j}{t}}$, we have the following lemma:
\begin{lemma}\label{binomcoeffmod4}
    When $t \neq 0, 2^{j-1}, 2^j$ and $j \geq 2$, then 
    \[\binom{2^j}{t}\equiv 0 \; \text{mod} \; 4.\]
\end{lemma}
\begin{proof}
    First to address the cases when $t \in \{0, 2^{j-1}, 2^j\}$, we know $\displaystyle{\binom{2^j}{0}=\binom{2^j}{2^j}=1}$ and $\displaystyle{\binom{2^j}{2^{j-1}}}\equiv 2 \; \text{mod} \; 4$. We prove the rest of the lemma by induction on $j$.
    
    \textbf{Basis Step} $\mathbf{j=2}$: 
    \[\binom{4}{1}=\binom{4}{3}=4. \]
    
    \textbf{Inductive Step:} Assume that $\displaystyle{\binom{2^{j-1}}{t}}\equiv 0 \; \text{mod} \; 4$ when $t \neq 0, 2^{j-2}, 2^{j-1}$. We start with when $0 <t<2^{j-2}$ and use the Chu-Vandermonde identity to see that $\displaystyle{\binom{2^j}{t}}$ is
    \[\sum_{i=0}^{t}\binom{2^{j-1}}{i}\binom{2^{j-1}}{t-i}=\binom{2^{j-1}}{0}\binom{2^{j-1}}{t}+\binom{2^{j-1}}{t}\binom{2^{j-1}}{0}+\sum_{i=1}^{t-1}\binom{2^{j-1}}{i}\binom{2^{j-1}}{t-i},\]
  which is congruent to $0 \; \text{mod} \; 4$  by induction. Next we take $2^{j-2}<t<2^{j-1}$; calculating $\displaystyle{\binom{2^j}{t}}$ yields
    \[\binom{2^{j-1}}{0}\binom{2^{j-1}}{t}+\binom{2^{j-1}}{t}\binom{2^{j-1}}{0}+\binom{2^{j-1}}{2^{j-2}}\binom{2^{j-1}}{t-2^{j-2}}\]\[+\binom{2^{j-1}}{t-2^{j-2}}\binom{2^{j-1}}{2^{j-2}}+ \smashoperator{ \sum_{\substack{i=1 \\ i \neq 2^{j-2}, t-2^{j-2}}}^{t-1}}\binom{2^{j-1}}{i}\binom{2^{j-1}}{t-i},\]
    which by induction is congruent to $0 \; \text{mod} \; 4$. Now if we take $t=2^{j-2}$, then $\displaystyle{\binom{2^j}{2^{j-2}}}$ is
    \[\binom{2^{j-1}}{0}\binom{2^{j-1}}{2^{j-2}}+\binom{2^{j-1}}{2^{j-2}}\binom{2^{j-1}}{0}+\sum_{i=1}^{2^{j-2}-1}\binom{2^{j-1}}{i}\binom{2^{j-1}}{t-i},\] which is equivalent to $ 0 \; \text{mod} \; 4$
    by Lemma \ref{middlebinomcoeff} and induction. 
    
    \indent Now suppose that $2^{j-1}<t<2^j$. Then $0<2^j-t<2^{j-1}$, so 
    \[\binom{2^j}{t}=\binom{2^j}{2^j-t} \equiv 0 \; \text{mod} \; 4\]
     and the lemma follows. 
    
\end{proof}

We still need to know more about the value of $\displaystyle{\binom{2^j}{2^{j-1}}}$ so we  prove the following lemma:
\begin{lemma}\label{middlebinomcoeffpt2}
    For $j \geq 2$, 
    \[\binom{2^j}{2^{j-1}} \equiv 6 \; \text{mod} \; 8.\]
\end{lemma}

\begin{proof}
    \textbf{Basis Step} $\mathbf{j=2}:$ 
    \[\binom{4}{2}=6. \]
    
    \textbf{Inductive Step:} Assume $\displaystyle{\binom{2^{j-1}}{2^{j-2}}}\equiv 6 \; \text{mod} \; 8$. By the Chu-Vandermonde Identity, $\displaystyle{\binom{2^j}{2^{j-1}}}$ is
    \[\sum_{i=0}^{2^{j-1}} \binom{2^{j-1}}{i}^2=\binom{2^{j-1}}{0}^2+\binom{2^{j-1}}{2^{j-1}}^2+\binom{2^{j-1}}{2^{j-2}}^2+\sum_{\substack{i=1 \\ i \neq 2^{j-2}}}^{2^{j-1}-1}\binom{2^{j-1}}{i}^2,\]
    
   which by induction and Lemma \ref{binomcoeffmod4} is $\equiv 6 \; \text{mod} \; 8$.
\end{proof}

Lemma \ref{alloddbinomcoeff} follows from Lucas's Theorem, which can be found in many books, including [10]:
\begin{lemma}\label{alloddbinomcoeff}
    $\displaystyle{\binom{2^j-1}{t}}$ is odd for every $0 \leq t \leq 2^j-1$.
\end{lemma}
To be more specific about $\displaystyle{\binom{2^j-1}{2^{j-1}}}$, we have the following lemma:

\begin{lemma}\label{binomcoeff_3mod4}
For $j\geq 2$,
    \[\binom{2^j-1}{2^{j-1}}\equiv 3 \; \text{mod} \; 4.\]
\end{lemma}
\begin{proof}
    We prove this via induction.
    
    \textbf{Basis Step} $j=2$: \[\binom{3}{2}=3.\]
    
    \textbf{Inductive Step:} Assume that $\displaystyle{\binom{2^{j-1}-1}{2^{j-2}}}\equiv 3 \; \text{mod} \; 4$. By the Chu-Vandermonde Identity, 
    \[\binom{2^j-1}{2^{j-1}}=\sum_{i=0}^{2^{j-1}}\binom{2^{j-1}}{i}\binom{2^{j-1}-1}{2^{j-1}-i}\] or
    \[\
    \binom{2^{j-1}}{0}\binom{2^{j-1}-1}{2^{j-1}}+\binom{2^{j-1}}{2^{j-1}}\binom{2^{j-1}-1}{0}+\binom{2^{j-1}}{2^{j-2}}\binom{2^{j-1}-1}{2^{j-2}}\]\[+\sum_{\substack{i=1\\i \neq 2^{j-2}}}^{2^{j-1}-1}\binom{2^{j-1}}{i}\binom{2^{j-1}-1}{2^{j-1}-i}.\]
    This is equivalent to
    $ 1+(2*3) \; \text{mod} \; 4
    \equiv 3 \; \text{mod} \; 4.$
\end{proof}
We now direct our attention back to the $a_{r,s}$ coefficients and examine what happens at certain values of $r$. We first note that for $n=2^k$, $a_{r,s+2^{k-1}}=a_{r,s-2^{k-1}}$.
\begin{lemma}\label{2coeffsumlemma1}
    Let $l \geq 1$ and $n=2^k$. Then for every $1\leq s\leq 2^k$,
    \[a_{l2^{k-1},s}+a_{l2^{k-1},s-2^{k-1}} \equiv 0 \; \text{mod} \; 2^l.\]
\end{lemma}
\begin{proof}
     We prove this by induction on $l$. Because of how long the subscripts on the $a_{r,s}$ coefficients will end up being in the next few proofs, we will be defining a few functions to represent certain coefficients. Let $f(\gamma, \delta)=a_{\gamma 2^{k-1},\delta}$. In terms of $f$, we want to show
     \[f(l,s)+f(l,s-2^{k-1}) \equiv 0 \; \text{mod} \; 2^l.\]
     
    \textbf{Basis Step} $\mathbf{l=1}$: It suffices to show $f(1,s)+f(1,s-2^{k-1}) \equiv 0 \; \text{mod} \; 2$. Then for $s \neq 1, 2^{k-1}+1$,
    \[f(1,s)+f(1,s-2^{k-1})=\binom{2^{k-1}}{s-1}+\binom{2^{k-1}}{s-2^{k-1}-1}\equiv 0 \; \text{mod} \; 2\]
    because both of these binomial coefficients are even. For $s=1, 2^{k-1}+1$,
    \[f(1,1)+f(1,1+2^{k-1})=\binom{2^{k-1}}{0}+\binom{2^{k-1}}{2^{k-1}}\equiv 0 \; \text{mod} \; 2\]
     and the basis case follows.
    
    \textbf{Inductive Step:} Assume $f(l-1,s)+f(l-1,s-2^{k-1}) \equiv 0 \; \text{mod} \; 2^{l-1}$, then \[f(l,s)+f(l,s-2^{k-1})\]
     is
   \[\sum_{i=1}^{2^k} [f(l-1,i)f(1,s-i+1)+f(l-1,i)f(1,s-2^{k-1}-i+1)].\]
   We now break this up into the following two sums:
    \[\sum_{i=1}^{2^{k-1}} [f(l-1,i)f(1,s-i+1)+f(l-1,i)f(1,s-2^{k-1}-i+1)]\]\[+\sum_{i=2^{k-1}+1}^{2^k} [f(l-1,i)f(1,s-i+1)+f(l-1,i)f(1,s-2^{k-1}-i+1)].\]
  Changing the indices of the second sum to match the first, this is  
    \[\sum_{i=1}^{2^{k-1}} [f(l-1,i)f(1,s-i+1)+f(l-1,i)f(1,s-2^{k-1}-i+1)]\]\[+\sum_{i=1}^{2^{k-1}} [f(l-1,i-2^{k-1})f(1,s-i+2^{k-1}+1)+f(l-1,i-2^{k-1})f(1,s-i+1)],\]
    which can be factored into
    \begin{equation}\label{Equation_lemma_1st_ars_sum}
    \sum_{i=1}^{2^{k-1}}[f(1,s-i+1)+f(1,s-i+2^{k-1}+1)][f(l-1,i)+f(l-1,i-2^{k-1})].
    \end{equation}

    By induction, since $f(l-1,i)+f(l-1,i-2^{k-1}) \equiv 0 \; \text{mod} \; 2^{l-1}$, the addend in the sum from  (\ref{Equation_lemma_1st_ars_sum}) is congruent to $0 \; \text{mod} \; 2^l$ if both $f(1,s-i+1)$ and $f(1,s-i+2^{k-1}+1)$ are even. The only times this does not happen is when $i=s$ and $i=s+2^{k-1}$. In both cases, the addend in the sum of (\ref{Equation_lemma_1st_ars_sum}) is
    \[[f(1,1)+f(1,2^{k-1}+1)][f(l-1,1)+f(l-1,2^{k-1}+1)].\]
    This is
    \[2[f(l-1,1)+f(l-1,2^{k-1}+1)],\]
    which is equivalent to $0 \; \text{mod} \; 2^l$
    by induction. Therefore, (\ref{Equation_lemma_1st_ars_sum}) is $\equiv 0 \; \text{mod} \; 2^l$ and the lemma follows.
\end{proof}
We have one last lemma to prove before Theorem \ref{MainThm}. We are once more interested in a sum of two specific $a_{r,s}$ coefficients.
\begin{lemma}\label{2coeffsumlemma2}
   Let $n=2^k$. For $l \geq 3$,
       \[a_{(l-1)2^{k-1},l2^{k-2}+1}+a_{(l-1)2^{k-1},l2^{k-2}-2^{k-1}+1} \equiv 0 \; \text{mod} \; 2^l.\]
   \end{lemma}
   \begin{proof}
   Let $g(\gamma,\epsilon ,\delta)=a_{\gamma 2^{k-1},\epsilon 2^{k-2}+\delta}$. Written in terms of $g$, we want to show
   \[g(l-1,l,1)+g(l-1,l,2^{k-1}+1) \equiv 0 \; \text{mod} \; 2^l.\]
       Calculating 
   $g(l-1,l,1)+g(l-1,l,2^{k-1}+1)$, we have
   \[\sum_{i=1}^{2^k} [f(1,i)g(l-2,l,2-i)+f(1,i)g(l-2,l,2^{k-1}+2-i)] .\]
   Like in the proof of Lemma \ref{2coeffsumlemma1}, we break the sum up and adjust the indices to give
    \[\sum_{i=1}^{2^{k-1}} [f(1,i)g(l-2,l,2-i)+f(1,i)g(l-2,l,2^{k-1}+2-i)]\]
        \[ +\sum_{i=1}^{2^{k-1}} [f(1,i-2^{k-1})g(l-2,l,2^{k-1}+2-i)+f(1,i-2^{k-1})g(l-2,l,2-i)] .\]
   Factoring yields
\begin{equation}\label{Equation_2nd_ars_sum}        
        \sum_{i=1}^{2^{k-1}}[f(1,i)+f(1,i-2^{k-1})][g(l-2,l,2-i)+g(l-2,l,2^{k-1}+2-i)].
        \end{equation}
      Note that by Lemma \ref{2coeffsumlemma1}, $g(l-2,l,2-i)+g(l-2,l,2^{k-1}+2-i) \equiv 0 \; \text{mod} \; 2^{l-2}$. Consequently, if  $f(1,i)+f(1,i-2^{k-1}) \equiv 0 \; \text{mod} \; 4$, then the whole addend in the sum of (\ref{Equation_2nd_ars_sum}) is congruent to $0 \; \text{mod} \; 2^l$.
      
        \indent Note that if $i \neq 1, 2^{k-2}+1, 2^{k-1}+1, 2^{k-1}+2^{k-2}+1$, then both $f(1,i)$ and $f(1,i-2^{k-1})$ are equivalent to $0 \; \text{mod} \; 4$. So reducing modulo $2^l$, Expression (\ref{Equation_2nd_ars_sum}) is
        \[[f(1,1)+f(1,2^{k-1}+1)][g(l-2,l,1)g(l-2,l,2^{k-1}+1)]\]
        \[+[f(1,2^{k-1}+1)+f(1,1)][g(l-2,l,2^{k-1}+1)+g(l-2,l,1)]\]
        \[+[g(1,1,1)+g(1,1,2^{k-1}+1)][g(l-2,l-1,1)+g(l-2,l-1,2^{k-1}+1)]\]
        \[+[g(1,1,2^{k-1}+1)+g(1,1,1)][g(l-2,l-1,2^{k-1}+1)+g(l-2,l-1,1)],\]
        which is congruent to $0 \; \text{mod} \; 2^l$ by Lemma \ref{2coeffsumlemma1} when $l-2 \geq 1$. 
        
   \end{proof}
   We now have all the tools we will use to prove our main theorem.
   \begin{proof}[Proof of Theorem \ref{MainThm}]
We once more note that in (I) on page 103 of \cite{Wong} and in Theorem 2.3 of \cite{Dular}, it is proved that all tuples vanish in $\mathbb{Z}_{2^l}^{2^k}$ for $l,k \in \mathbb{Z}^+$. We also provide a proof of this because to prove that $L_{2^l}(2^k)=(l+1)2^{k-1}$, we will show $D^{(l+1)2^{k-1}}(0,0,...,0,1)=(0,0,...,0)$, which will give us that $P_{2^l}(2^k)=1$ and all tuples in $\mathbb{Z}_{2^l}^{2^k}$ vanish.
 
     \indent We prove this via induction on $l$ with two basis cases. However, we will often need to use theorems that rely on $k \geq 2$ in our proof, so we first prove our theorem for $k=1$ or that $L_{2^l}(2)=l+1$ and that $P_{2^l}(2)=1$ for every $l>0$.
     
     \textbf{Basis Step} $\mathbf{k=1}$: Notice that $D^{\alpha}(0,1)=(2^{\alpha-1},2^{\alpha-1})$. This means that $D^{l+1}(0,1)=(0,0)$ and that $P_{2^l}(2)=1$. Since we also have that 
     \[D^l(0,1)=(2^{l-1},2^{l-1})\not \equiv (0,0) \; \text{mod} \; 2^l,\] $L_{2^l}(2)=l+1$ follows.
       
       \indent For the rest of the proof, we assume $n=2^k$, $k \geq 1$.
       
\textbf{Basis Step} $\mathbf{l=1}$: We first  prove $f(2,s) \equiv 0 \; \text{mod} \; 2$ for all $1 \leq s \leq n$. Recall that $f(2,1)=2$ and for $s \neq 1$, $f(2,s)=\displaystyle{\binom{2^k}{s-1}} \equiv 0 \; \text{mod} \; 2$ which gives us $P_2(2^k)=1$. Note also that for $r<2^k$, $a_{r,1}=1 \not \equiv 0 \; \text{mod} \; 2$. Therefore, this is the first time we have that $a_{r,s} \equiv 0 \; \text{mod} \; 2$ for all $s$, so $L_2(2^k)=2^k$.

    \textbf{Basis Step} $\mathbf{l=2}$: It suffices to show that $L_4(2^k)=3*2^{k-1}$ and $P_4(2^k)=1$. We begin by showing that $f(3,s) \equiv 0 \; \text{mod} \; 4$ for every s. Start by noting
    \[f(3,s)=\sum_{i=1}^{2^k}f(1,i)f(2,s-i+1)\]
     and then separate the terms where $i=1,2^{k-1}+1$ to produce
    \begin{equation}\label{Equation_basisl=2_first}
    f(1,2^{k-1}+1)f(2,s-2^{k-1})+f(1,1)f(2,s)+\smashoperator{\sum_{\substack{i=2\\i\neq 2^{k-1}+1}}}f(1,i)f(2,s-i+1) .
    \end{equation}
   We do this because $f(1,i) \equiv 0 \; \text{mod} \; 2$ for $i \neq 1, 2^{k-1}+1$ and $f(2,s-i+1) \equiv 0 \; \text{mod} \; 2$ as shown in our first basis case. So now Expression (\ref{Equation_basisl=2_first}) is equivalent to
    \begin{equation}\label{Equation_basis_Period}
     f(2,s-2^{k-1})+f(2,s) \; \text{mod} \; 4,
    \end{equation}
    which, because of Lemma \ref{2coeffsumlemma1}, is congruent to $0 \; \text{mod} \; 4$.
      Therefore we have that $P_4(2^k)=1$ and $L_4(2^k)\leq 3*2^{k-1}$. So we now need to show that there exists $s$ such that $a_{3*2^{k-1}-1,s} \not \equiv 0 \; \text{mod} \; 4$. Define $h(\gamma, \delta)=a_{\gamma 2^{k-1}-1,\delta}$
      Calculating $h(3,s)$ for general $s$, we have
    \[h(3,s)=\sum_{i=1}^{2^k} f(1,i)h(2,s-i+1) .\]
    We now separate the terms where $i$ is $1, 2^{k-2}+1$ or $2^{k-1}+1$, which gives
    \begin{equation}\label{Equation_basis_length_long}
    f(1,1)h(2,s)+f(1,2^{k-1}+1)h(2,s-2^{k-1})+g(1,1,1)h(2,s-2^{k-2})
    \end{equation}
    \[+\sum_{i \in J}f(1,i)h(2,s-i+1)\]
    where $J=\{2 \leq i \leq 2^k \; | \; i \neq 2^{k-1}, 2^{k-2}+1\}$. We do this so the sum over $J$ is congruent to $0 \; \text{mod} \; 4$ by Lemma \ref{binomcoeffmod4}.
    If we also take $s=1$, Expression (\ref{Equation_basis_length_long}) is congruent to
   \[
      f(1,1)h(2,1)+f(1,2^{k-1}+1)h(2,2^{k-1}+1)+g(1,1,1)h(2,2^k-2^{k-2}+1) \; \text{mod} \; 4,
     \]
     which is equivalent to
    $ 1+h(2,2^{k-1}+1)+2 \; \text{mod} \; 4$
    because $g(1,1,1) \equiv 2 \; \text{mod} \; 4$ and $h(2,2^k-2^{k-2}+1)$ is odd. From Lemma \ref{binomcoeff_3mod4}, this is congruent to $ 2 \; \text{mod} \; 4$ and $h(3,1) \not \equiv 0 \; \text{mod} \; 4$, so $L_4(2^k)=3*2^{k-1}$.
   
    \textbf{Inductive Step:} Assume that $L_{2^{l-1}}(2^k)=l2^{k-1}$ and $P_{2^{l-1}}(2^k)=1$. This implies that $f(l,s) \equiv 0 \; \text{mod} \; 2^{l-1}$ for every $s$. Calculating $f(l+1,s)$, we get
    \[f(l+1,s)=\sum_{i=1}^{2^k}f(l,i)f(1,s-i+1) .\]
    If we separate out the terms where $s-i+1=1, 2^{k-1}+1$, this becomes
    \begin{equation}\label{Equation_inductive_justperiod}
    f(l,s)f(1,1)+f(l,s-2^{k-1})f(1,2^{k-1}+1)+\sum_{i \in J}f(l,i)f(1,s-i+1)
    \end{equation}
    where $J=\{1 \leq i \leq 2^k \; | \; i \neq s, s-2^{k-1}\}$. We do this so the sum over $J$ is equivalent to $0 \; \text{mod} \; 2^{l}$ by induction and because $f(1,s-i+1)$ is even over $J$. Therefore, Expression (\ref{Equation_inductive_justperiod}) is equivalent to
   \[
     f(l,s)+f(l,s-2^{k-1}) \; \text{mod} \; 2^l,
    \]
   which by Lemma \ref{2coeffsumlemma1} is
    $\equiv 0 \; \text{mod} \; 2^l$. This gives us that $P_m(n)=1$ and 
    \[L_m(n) \leq (l+1)2^{k-1}.\] Now showing $L_m(n)> (l+1)2^{k-1}-1$ will prove the rest of the theorem. Note that because $P_m(n)=1$, it suffices to show that there exists $s$ such that 
    \[h(l+1,s) \not \equiv 0 \; \text{mod} \; 2^l.\]
     We start by breaking down $h(l+1,s)$ as follows:
     \begin{equation}\label{Equation_inductive_first}
     h(l+1,s)=\sum_{i=1}^{2^k}f(l,i)h(1,s-i+1) .
     \end{equation}
     Because $a_{r,s}=a_{r,r-s+2}$ by Lemma \ref{a_rs_when_=}, most coefficients have another coefficient that it is equal to. The case where they do not is when $s\equiv r-s+2 \; \text{mod} \; 2^k $. For our case then, $h(l,i)$ is not equal to another coefficient when 
     \[i=l2^{k-1}-i+2\]
      and 
      \[i=l2^{k-1}-i+2^k+2.\] 
      Solving for $i$, this is when
     $i=l2^{k-2}+1$ and $i=l2^{k-2}+2^{k-1}+1.$
     We will separate these out from our sum so, using Equation (\ref{Equation_inductive_first}), we can view $h(l+1,s)$ like
     \[g(l,l,1)h(1,s-l2^{k-2})+g(l,l,2^{k-1}+1)h(1,s-l2^{k-2}-2^{k-1})\]
     \[+\sum_{i \in M}[f(l,i)h(1,s-i+1)+g(l,2l,2-i)h(1,s-l2^{k-1}+i-1)\]
     where $M$ is defined to preserve equality. We now take $s=l2^{k-2}+1$, so  $h(l+1,l2^{k-2}+1)$ is
     \begin{equation}\label{Equation_inductive_long}
    g(l,l,1)h(1,1)+g(l,l,2^{k-1}+1)h(1, 2^{k-1}+1)  
    \end{equation}    
    \[ +\sum_{i \in M}[f(l,i)h(1,l2^{k-2}+2-i)+g(l,l,2-i)h(1,l2^{k-2}-l2^{k-1}+i)].\]
     
     Taking a look at each piece of (\ref{Equation_inductive_long}), $g(l,l,2^{k-1}+1)h(1, 2^{k-1}+1)=0$ because $h(1,2^{k-1}+1)=0$. For the sum in (\ref{Equation_inductive_long}), this is equal to 
     \[\sum_{i \in M}f(l,i)[h(1,l2^{k-2}+2-i)+h(1,l2^{k-2}-l2^{k-1}+i)], \]
     because $f(l,i)=g(l,l,2-i)$. Since $h(1,s)$ is odd for all $s$ by Lemma \ref{alloddbinomcoeff}, 
     \[h(1,l2^{k-2}-i+2)+h(1,l2^{k-2}-l2^{k-1}+i)\]
      is even. We know $f(l,i) \equiv 0 \; \text{mod} \; 2^{l-1}$, so the whole sum is $\equiv 0 \; \text{mod} \; 2^l$. Therefore, reducing Expression (\ref{Equation_inductive_long}) modulo $2^l$, we conclude that
     \[g(l+1,l,1)\equiv g(l,l,1) \; \text{mod} \; 2^l.\]
     So if we can prove that $g(l,l,1) \not \equiv 0 \; \text{mod} \; 2^l$, then the theorem will follow. Since we already know that $g(l,l,1) \equiv 0 \; \text{mod} \; 2^{l-1}$, we need to prove the following claim:
     
      \textbf{Claim:} $g(l,l,1) \equiv 2^{l-1} \; \text{mod} \; 2^l$ for $l \geq 2$.
     
     We prove this claim via induction $l$.
     
     \textit{Basis Step} $\mathit{l=2}$:  $g(2,2,1) \equiv 2 \; \text{mod} \; 4$ by Lemma \ref{middlebinomcoeff} and the basis case follows.
     
     \textit{Inductive Step:} Assume now that $g(l-1,l-1,1) \equiv 2^{l-2} \; \text{mod} \; 2^{l-1}$, then $g(l,l,1)$ is
     \[\sum_{i=1}^{2^k}f(l-1,i)g(1,l,2-i) .\]
     Let $J^*=\{1 \leq i \leq 2^k \; | \; i \neq (l-1)2^{k-2}+1, l2^{k-2}+1, l2^{k-2}-2^{k-1}+1\}$ and separate all terms not in $J^*$ to see this is
\begin{equation}\label{Equation_inductive_claim}     
     g(l-1,l-1,1)g(1,1,1)+g(l-1,l,1)f(1,1)+g(l-1,l,2^{k-1}+1)g(1,2,1)
     \end{equation}
     \[+\sum_{i \in J^*}f(l-1,i)g(1,l,2-i)\]
     Once more, we separate this sum over $J^*$ because it is $\equiv 0 \; \text{mod} \; 2^l$ because 
     \[f(l-1,i) \equiv 0 \; \text{mod} \; 2^{l-2}\]
      and $g(1,l,2-i) \equiv 0 \; \text{mod} \; 4$ by Lemma \ref{binomcoeffmod4}.
 Looking at the remaining pieces from (\ref{Equation_inductive_claim}), $g(l-1,l-1,1)g(1,1,1) \equiv 2^{l-1} \; \text{mod} \; 2^l$ by induction and Lemma \ref{middlebinomcoeff}. Next, 
     \[g(l-1,l,1)f(1,1)+g(l-1,l,2^{k-1}+1)g(1,2,1)=g(l-1,l,1)+g(l-1,l,2^{k-1}+1)\]
     
    which by Lemma \ref{2coeffsumlemma2} is $\equiv 0 \; \text{mod} \; 2^l$. 
     The claim follows.
     
     \indent Therefore, $g(l+1,l,1) \equiv 2^{l-1} \; \text{mod} \; 2^l$ and $L_m(n)>(l+1)2^{k-1}-1$. Hence, $L_m(n)=(l+1)2^{k-1}$. 
   \end{proof}

\end{document}